\documentclass[final,onefignum,onetabnum]{siamart171218}
\usepackage{amssymb,latexsym,amsmath,amsxtra,amsfonts,color}
\usepackage{graphicx}
\usepackage{parskip}
\usepackage{mathrsfs}
\usepackage[title]{appendix}

\newtheorem{thm}{Theorem}[section]

\newtheorem{prop}[thm]{Proposition}

\newtheorem{define}[thm]{Definition}

\newtheorem{example}[thm]{Example}

\newtheorem{remarks}[thm]{Remarks}

\DeclareMathOperator{\ran}{ran}
\DeclareMathOperator{\dom}{dom}

\begin{document}

\title{Applications of a class of  Herglotz operator pencils.}

\author{Andrea K. Barreiro \footnote{Department of Mathematics, Southern
    Methodist University} \\
 Jared C. Bronski\footnote{Department of Mathematics, University of Illinois} \\ Zoi Rapti\footnote{Department of Mathematics, University of Illinois}  }

\date{\today}

\maketitle

\begin{abstract}
We identify a  class of operator pencils, arising in a number of
 applications, which have only real
eigenvalues. In the one-dimensional case we prove a novel version of
the Sturm oscillation theorem: if the dependence on the eigenvalue
parameter is of degree $k$ then the real axis can be partitioned into
a union of $k$ disjoint intervals, each of which enjoys a Sturm
oscillation theorem: on each interval there is an increasing sequence
of eigenvalues that are indexed by the number of roots of the
associated eigenfunction.  One consequence of this is that it
guarantees that 
the spectrum of these
operator pencils has finite accumulation points, implying that the
operators 
do not have compact resolvents.  
As an application we apply this theory to an epidemic model and several species dispersal models arising in biology.
\end{abstract}

\section{Introduction}
\label{sec:intro}

There are a number of contexts in which  eigenvalue pencils 
arise in applied mathematics \cite{kollar2014}. 
When studying the stability of standing waves and other coherent structures 
it is frequently the case that one must consider a non-self-adjoint
eigenvalue pencil. In the context of integrable systems the scattering
problem that diagonalizes an integrable  nonlinear flow frequently
takes the form of a eigenvalue pencil; take, for instance the
Sine-Gordon equation in the laboratory frame \cite{Takhtajan}  or the derivative
nonlinear Schr\"odinger equation \cite{KaupNewell}. Additionally we
will present a number of other problems arising in the analysis of
biological problems where rational operator pencils arise.     
The difficulty with these problems, from the point of view of 
both analysis and numerical experimentation, is that one 
has no apriori information as to where the 
spectrum of the problem might lie. The purpose of this note is to
point out a class of pencils, arising in a number of applications,
which only admit real eigenvalues, namely those that depend on the
spectral parameter in a Herglotz way. We show that a
number of the properties of self-adjoint eigenvalue problems carry
over to Herglotz operator pencils: the eigenvalues are real and
semi-simple, and for second order rational Herglotz pencils we have 
an interesting generalization of the Sturm oscillation theorem. We begin 
by reminding the reader of the definition of a Herglotz function
\cite{Simon}. These functions are also sometimes referred to as
Nevanlinna or Pick functions, but in the spectral theory literature
they are most often referred to as Herglotz functions, so that is the
usage that we will adopt here. 

\begin{define}
A meromorphic function $f(\lambda)$ is Herglotz if $\operatorname{Im}(\lambda)>0$ (resp. $\operatorname{Im}(\lambda)<0$) implies that $\operatorname{Im}(f(\lambda))>0$ (resp. $\operatorname{Im}(f(\lambda))<0$))
\end{define}     
The following is standard:
\begin{prop}
If $f(\lambda)$ is Herglotz then 
\begin{itemize}
\item The zeroes and poles of $f(\lambda)$ lie on the real axis.
\item If $\lambda$ is real and not a pole then
$f'(\lambda)>0$. 
\item The zeroes and poles of $f(\lambda)$ alternate on the real axis. 
\end{itemize}
\end{prop}

For the purposes of this paper the canonical example of a Herglotz function 
is the rational function 
\[
f(\lambda) = \sum_{i=1}^N \frac{A_i}{\alpha_i - \lambda} + B + C \lambda
\]
with $A_i,B,C$ real and positive and $\alpha_i$ real, as all of the
problems we consider lead to rational operator pencils.   More
generally one has a representation formula for any Herglotz function 
\[
f(\lambda) = B + C \lambda + \int_{\mathbb R} \left( \frac{1}{x-\lambda} -
\frac{\lambda}{1+\lambda^2} \right) d\mu(x)
\] 
for some measure $d\mu(x)$ but we are unaware of any applications that require
this general representation. 

\section{Motivating Examples}
\label{sec:examples}
In this section we present some examples intended to motivate  the later
discussion.  As a first example we consider a reaction-diffusion equation of the form 
\begin{align}
& u_t = u_{xx} + f(u,v) \\
& v_t = g(u,v) 
\label{eqn:ReactDiffuse}
\end{align}
together with appropriate boundary conditions. Assuming that we can
find a steady state solution 
\begin{align}
&u(x,t) = \bar u(x) \\
&v(x,t) = \bar v(x)
\end{align}
we consider the linearization of the evolution around the steady state
$(\bar u, \bar v)$. This leads to the following equation for the
tangent flow 
\begin{align}
&p_t = p_{xx} + \frac{\partial f}{\partial u}(\bar u,\bar v) p +\frac{\partial f}{\partial v}(\bar u,\bar v) q  \\
& q_t = \frac{\partial g}{\partial u}(\bar u,\bar v) p +\frac{\partial g}{\partial v}(\bar u,\bar v) q 
\label{eqn:ReactDiffuseLSA}
\end{align}
or, equivalently the linearized stability problem
\begin{align}
&\lambda p = p_{xx} + \frac{\partial f}{\partial u}(\bar u,\bar v) p
  +\frac{\partial f}{\partial v}(\bar u,\bar v) q  
  \label{eqn:LinStab1} \\
& \lambda q = \frac{\partial g}{\partial u}(\bar u,\bar v) p +\frac{\partial g}{\partial v}(\bar u,\bar v) q 
\label{eqn:LinStab2}
\end{align}
This is an ordinary, though non-selfadjoint eigenvalue problem, but
after one eliminates the function $q$ it becomes a rational eigenvalue
pencil. 
\[
\lambda p = p_{xx} + \frac{\partial f}{\partial u}(\bar u,\bar v) p
+\frac{\partial f}{\partial v}(\bar u,\bar v)  \frac{\partial
  g}{\partial u}(\bar u,\bar v) \frac{1}{\lambda - \frac{\partial g}{\partial v}(\bar u,\bar v) }
\]

The main observation is that under the sign condition $f_v g_u(\bar
u,\bar v)\geq 0$ the right hand side of the operator pencil is (for $x$
fixed), a Herglotz function which implies some very nice behavior in
the eigenvalue problem, including (as we shall see) reality of the spectrum and
semi-simplicity of the eigenvalues. This generalizes the case when
$f_v=g_u$, for which case the flow is a gradient and the second variation is
obviously self-adjoint.  

One could object that the original system (\ref{eqn:LinStab1},
\ref{eqn:LinStab2}) is, in fact, self-adjoint under the modified inner
product $\Vert (p,q)\Vert_2^2 = \int v^2 + \frac{f_w}{g_v}w^2$ when
the sign condition $f_w g_v>0$ is met. In some sense this must happen,
since it is an old theorem of Drazin and  Haynsworth \cite{Haynsworth1962} that any matrix
having all real eigenvalues and $n$ linearly  independent eigenvectors
is self-adjoint under some appropriate inner product.  However, in
general it is unclear how to check whether there exists an
inner product rendering a given operator self-adjoint. Checking that a
given function is Herglotz, on the other hand, is relatively
straightforward, and amounts (for a rational pencil) to checking a
finite
number of reality/sign conditions.  

One biological model that falls into this class is the following system modeling
the spread of a human disease through infective propagules \cite{capasso1997} 
\begin{eqnarray}
u_t &=& d u_{xx} - a_{11} u+ a_{12} v = d u_{xx} + f(u, v) \\
v_t &=&  \tilde{g}(u)- a_{22} v = g(u, v). 
\end{eqnarray}
Here, $u$ denotes the spatial density of the infectious propagules, $v$ denotes the 
density of the human infective class; the parameters $d, a_{ij} >0$; the function
$\tilde{g}: \mathbb{R}_+ \to \mathbb{R}_+ $ and 
$\tilde{g}'(u)>0$ for $u>0$.
Then, the corresponding eigenvalue problem reads 
\[ d p_{xx} -a_{11} p = \lambda p - a_{12} \tilde{g}' \frac{1}{\lambda + a_{22}} p. \]
Since $a_{12} \tilde{g}'>0 $ the condition for the reality of the spectrum follows. 

Equations of similar form also arise in combustion, typically solid
combustion. In this situation heat is free to diffuse but the solid
reactant does not diffuse, at least on combustion time-scales. For
work on this problem we refer the reader to the article of Ghazaryan,
Latushkin and Schechter \cite{Ghazaryan2010a, Ghazaryan2010b}, but we do note that in
such combustion problems the sign condition is typically not
  met: in the combustion context one expects $f_w g_u<0.$

Similarly  we can consider a system with three reacting species, one of
which is allowed to diffuse 
\begin{eqnarray}
&& u_t = u_{xx} + f(u,v,w) \nonumber \\
&& v_t = g(u,v,w) \label{eq:twodiff} \\
&& w_t = h(u,v,w). \nonumber
\end{eqnarray}
Linearizing about a steady state solution yields the following eigenvalue
problem
\begin{eqnarray}
&& \lambda p = p_{xx} + f_u p +f_v q + f_w s \nonumber \\
&& \lambda q = g_u p + g_v q + g_w s\nonumber \\
&& \lambda s = h_u p + h_v q + h_w s \nonumber
\end{eqnarray}
Solving the last two for $q, s$ in terms of  $p$ results in
\[q = \frac{g_u(\lambda - h_w) + h_u g_w }{(\lambda - g_v)(\lambda - h_w) - g_w h_v} p,~~
s = \frac{(\lambda - g_v) h_u + h_v g_u }{(\lambda - g_v)(\lambda - h_w) - g_w h_v} p\]
and substituting into the first equation yields
\begin{eqnarray*}
&& p_{xx} + f_u p = \left( \lambda - 
f_v \frac{g_u(\lambda - h_w)+ h_u g_w}{(\lambda - g_v)(\lambda - h_w) - g_w h_v} - 
f_w \frac{h_u(\lambda - g_v)+ h_v g_u}{(\lambda - g_v)(\lambda - h_w)-g_w h_v} \right) p 
\equiv H(\lambda) p \\
&& p_{xx} + f_u p = \left( \lambda - \frac{\alpha(x) \lambda + \beta(x)
   }{\lambda^2 + \gamma(x) \lambda + \delta(x)} \right)
\end{eqnarray*}
The function $\lambda - \frac{\alpha(x) \lambda + \beta(x)
   }{\lambda^2 + \gamma(x) \lambda + \delta} $ is (for fixed $x$) a
   Herglotz function of $\lambda$ if and only if the following
conditions are met 
\begin{itemize}
\item The roots of $\lambda^2 + \gamma(x) \lambda + \delta(x)$ are
  real.
\item The residues of the function $\frac{\alpha(x) \lambda + \beta(x)
   }{\lambda^2 + \gamma(x) \lambda + \delta}$ at the roots are
   positive.
\end{itemize}
 The above two conditions translate into the following sign conditions
 on the coefficients.
\begin{itemize}
\item $\beta \gamma - \alpha \delta >0$
\item $\beta^2 - \alpha (\beta \gamma - \alpha \delta)<0$
\end{itemize}
Equivalently, reality of  the spectrum follows if 
\begin{eqnarray} 
\left| \begin{array}{cc} f_v & f_w \\
g_v & g_w  \end{array} \right| \left| \begin{array}{cc} g_u & g_v \\
h_u & h_v  \end{array} \right| + \left| \begin{array}{cc} f_v & f_w \\
h_v & h_w  \end{array} \right| \left| \begin{array}{cc} g_u & g_w \\
h_u & h_w  \end{array} \right| >0 ~~ \mbox{and} \nonumber \\
\left( h_u \left| \begin{array}{cc} g_u & g_w \\
h_u & h_w  \end{array} \right| + g_u \left| \begin{array}{cc} g_u & g_v \\
h_u & h_v  \end{array} \right| \right) \left(  f_v \left| \begin{array}{cc} f_v & f_w \\
g_v & g_w  \end{array} \right| + f_w \left| \begin{array}{cc} f_v & f_w \\
h_v & h_w  \end{array} \right|\right)>0
\nonumber
\end{eqnarray} 


A system in the form  of (\ref{eq:twodiff}) is considered in \cite{feng2013} 
as a model of the dynamics of two plants $v, w$ and one herbivore $u$
\begin{eqnarray*}
u_t &=& d u_{xx} +\left(b_1 s_1 v + 
b_2 s_2 w \left(1 - \frac{s_2 w}{4 n (1+h_1 s_1 v+ h_2 s_2 w)} \right) - m \right) 
\frac{u}{1+h_1 s_1 v+ h_2 s_2 w} \\
&=& d u_{xx} + f(u, v, w) \\
v_t &=&  r_1 v \left( 1 -(v+a_{12} w) \right) - 
\frac{s_1 u v}{1+h_1 s_1 v+ h_2 s_2 w}  = g(u, v, w), \\
w_t &=&  r_2 w \left( 1 -(w+a_{21} v) \right) - 
\frac{s_2 u w}{1+h_1 s_1 v+ h_2 s_2 w} 
\left(1 - \frac{s_2 w}{4 n (1+h_1 s_1 v+ h_2 s_2 w)} \right)  = h(u, v, w),
\end{eqnarray*}
where $d, a_{ij} >0$ and all other parameters are positive.

In \cite{lou2004} three models of morphogen concentrations were analyzed. 
Nonlinear reaction-diffusion equations  were used to model
morphogen  gradients that arise in the patterning of {\it Drosophila} wings.
In a series of works, the authors showed that the linear
stability analysis of steady state solutions yields a nonlinear eigenvalue
problem. They also demonstrated that the spectrum is real and stable.
In this section, we will offer an alternative proof  of these facts using
our main result.

Following the notation in \cite{lou2004}, the dimensionless normalized 
concentration of the diffusing morphogen species is denoted with $a(x, t)$ and 
that of the morphogens that are bound to receptors is denoted by $b(x, t)$.
Then, the system reads
\begin{eqnarray}
&& \frac{\partial a}{\partial t} = \frac{\partial^2 a}{\partial x^2} -h_0 a(1-b) +f_0 b,
~~ x \in (0, 1),~~ t>0, ~~ \nonumber \\
&& \frac{\partial a(0,t)}{\partial t} = \nu_0 -h_o a(0,t) (1-b(0,t)) + f_0 b(0,t),~~ a(1,t)=0 \label{dmorph} \\
&& \frac{\partial b}{\partial t} = h_0 a(1-b)  - (f_0 +g_0) b,
~~ x \in [0, 1],~~ t>0. \label{bmorph} 
\end{eqnarray}
Here, the constants $h_0, f_0, g_0, \nu_0$ are positive and correspond to 
normalized rates.


If  one  linearizes about the time-independent steady-state solution 
$\left( \bar{a}(x), ~\bar{b}(x) \right)$ considering solutions of  the form 
\[a(x,t) = \bar{a}(x) + e^{\lambda t} \hat{a}(x), ~~
b(x,t) = \bar{b}(x) +e^{\lambda t} \hat{b}(x,t),\]

and after some standard algebra to eliminate the algebraic equation for the bound 
morphogen $b$, the following nonlinear eigenvalue problem is obtained
\begin{eqnarray}
\hat{a}_{xx} - \left( \lambda + q(x, \lambda) \right) \hat{a} =0, ~~\mbox{where}~~
q(x, \lambda) = \frac{h_0 (f_0 +g_0)}{f_0 + g_0 + h_0 \bar{a}} 
\frac{\lambda + g_0}{\lambda + f_0 + g_0 + h_0 \bar{a}},
\label{nonlinmorph}
\end{eqnarray} 
together with appropriate boundary conditions.


One easily notices that $q(x, \lambda)$ is a Herglotz function, since 
\[f_0+g_0+h_0 \bar{a} -g_0 = f_0 + h_0 \bar{a} >0.\] Therefore,
\[Q(x, \lambda) = \lambda + q(x, \lambda)\] is a Herglotz function as well.
Using our main result it then follows  that
the spectrum is real. 

As these examples illustrate, a wealth of biological models have similar structure.
Their common theme is that some degrees of freedom permit diffusion or mobility, while others do not.
In many compartmental systems arising in biological applications, 
some of the components are mobile (such as herbivores and pollinators) while others 
are not (such as plants) \cite{feng2013, lewis1994, morris1997, sanchez2011}. 
In epidemic compartmental models, such distinctions between mobile and 
motionless compartments also arise \cite{capasso1997, murray1986, wang2012} due 
to the fact that infected animals are mainly responsible for spreading the disease, 
or when infective propagules disperse randomly while the infected  population is 
characterized by small  mobility.  Another setting where such models arise is in pattern formation, where diffusing
morphogens interact with others that are bound to cell receptors \cite{lou2004}.

The final example is more theoretical, and is motivated by work of Kapitula and Promislow \cite{kapitula2012} on
the so-called Krein matrix method. Consider a self-adjoint eigenvalue
problem that can be written in the following block-partitioned form
\[
({\bf H}-\lambda) v = \left(\begin{array}{cc} {\bf A}-\lambda & {\bf B} \\ {\bf B}^t &
                                                                     {\bf
                                                                     C}-\lambda \end{array}\right)
                                                               v = 0,
\]
with ${\bf A} = {\bf A}^t$ and ${\bf C} = {\bf C}^t$. 
The main idea is that, if we algebraically eliminate certain degrees
of freedom, we are naturally led to a Herglotz pencil. To see this
note that, if $\lambda$ is in the resolvent set of ${\bf C}$ then we
have the following identity.
\begin{align*}
{\bf U} {\bf H} {\bf U}^t &=    
\left(\begin{array}{cc} {\bf I} & -{\bf B}({\bf C}-\lambda)^{-1} \\ 
0 &   {\bf I} \end{array}\right)
\left(\begin{array}{cc} {\bf A} & {\bf B} \\
 {\bf B}^t & {\bf C} \end{array}\right)
  \left(\begin{array}{cc} {\bf I} & 0 \\ 
-({\bf C}-\lambda)^{-1}{\bf B}^t & {\bf I} \end{array}\right)\\
&=
 \left(\begin{array}{cc}{\bf A}- \lambda -{\bf B}^t ({\bf C}-\lambda)^{-1}{\bf B} &0 \\
0&  ({\bf C}-\lambda) \end{array}\right)
\end{align*}
Therefore for $\lambda$ in the resolvent set of ${\bf C}$ the operator
${\bf U}$ is bounded and invertible and the
eigenvalues of the full operator are eigenvalues of the Herglotz
operator pencil ${\bf A} - \lambda - {\bf B}^t ({\bf C}-\lambda)^{-1}
{\bf B}.$ If $\lambda$ is in the spectrum of ${\bf C}$ (the complement
of the resolvent set) one must do
additional work to determine whether or not $\lambda$  is an eigenvalue of the full
problem -- see the work of Kapitula and Promislow for details. This
example shows that Schur reduction of a self-adjoint eigenvalue
problem leads naturally to a Herglotz operator pencil, and further motivates
the idea that Herglotz pencils should behave
like self-adjoint eigenvalue problems. 

\section{Main Results}
\label{sec:results}

Our basic observation  is that the
standard elementary argument that the spectrum of a self-adjoint
operator is real carries over in a natural way to operator pencils
that depend on the eigenvalue in a Herglotz way.  The main claim of this paper is 
\begin{thm}
Suppose that we have an eigenvalue pencil of the form 
\[
{\bf H}\psi = \left( A_0 \lambda - \sum_{k=1} \frac{A_i}{\lambda - \alpha_i} \right) \psi
\]
where ${\bf H}$ is a self-adjoint operator on some domain $\dom({\bf H})$, and ${\bf A}_i$ are
self-adjoint positive definite operators on  $\dom({\bf H})$.  Then the eigenvalue pencil has only real
eigenvalues. Furthermore the eigenvalues are semi-simple. 
\end{thm}
\begin{proof}
Suppose that $\lambda$ is an eigenvalue with eigenfunction $\psi$. 
Taking the inner product of both sides with $\psi$ we have that 
\[
\langle\psi, {\bf H} \psi \rangle = \langle \psi, {\bf A}_0
\psi\rangle \lambda - \sum_{k=1} \frac{\langle \psi,{\bf A}_k \psi
  \rangle}{\lambda - \alpha_i}
\]
From the self-adjointness of ${\bf H}$ we have that the left-hand side is 
purely real; the righthand side is a Herglotz function. 
Thus, if $\psi(x)$ is an eigenfunction the $\lambda $ is a root of some
Hergoltz function $\tilde g_\psi(\lambda)$ and therefore must be real. 

To see the semi-simplicity of the eigenvalues note that the
condition for the existence of a Jordan chain for an operator pencil
is the existence of a solution to 
\begin{align*}
&{\bf M}(\lambda_0) \psi_0  = 0& \\
&{\bf M}(\lambda_0) \psi_1 +   \frac{d{\bf M}}{d\lambda}(\lambda_0) \psi_0 = 0& 
\end{align*}
with $\psi_0$ non-zero. Taking the inner product of the second
equation  with $\psi_0$ gives 
\[
\langle \psi_0, \frac{d{\bf M}}{d\lambda}(\lambda_0) \psi_0 \rangle =0,
\]
but the Herglotz nature of the operator implies that the operator
$\frac{d{\bf M}}{d\lambda}(\lambda_0) $ is positive definite. 

\end{proof}

Next we specialize to the rational Sturm-Liouville problem 
\begin{equation}
-(D(x) p_{x})_x = \lambda W_0(x) p - \sum_{i=1}^N \frac{W_i(x)}{\lambda - \alpha_i}
- V(x) p = g(x, \lambda) p; \quad p_x(0) = 0 = p_x(L)
\label{eqn:RationalSturmLiouville}
\end{equation}

As motivation for the next result we note that in the special case in
which $V(x),W_i(x)$ and $D(x)$ are all constant we have that the
solution is given by 
\[
p(x) = \cos \sqrt\frac{g(\lambda)}{D}x 
\] 
with the eigenvalue condition $g(\lambda)=k^2 \pi^2 D$. Since
$g(\lambda)$ is Herglotz it is monotone and takes every real value for
$\lambda \in (\alpha_{i-1},\alpha_i)$. Thus we have that for each such
interval $(\alpha_{i-1},\alpha_i)$ there is a sequence of
eigenfunctions $p_k(x)$ which are indexed by $k$, the number of roots
in $ (\alpha_{i-1},\alpha_i)$.   

Equation (\ref{eqn:RationalSturmLiouville}) will typically not, of course, be
solvable in closed form but in the WKB approximation, which is valid
when $g(x, \lambda)$ is large (which occurs for $\lambda=\alpha_I^-$)
we have the approximate eigenfunction
 \begin{equation}
p_k(x,\lambda) \approx \frac{\cos \left( \int_0^x
  \sqrt\frac{g(x, \lambda)}{D(x)} \right)}{\sqrt{D(x)} (g(x, \lambda))^{\frac14}}
\label{eqn:WKB}
\end{equation}
where the WKB quantization condition is that 
\begin{equation}
\int_0^L
\sqrt{\frac{g(x, \lambda)}{D(x)}}dx = k \pi.
\label{eqn:Quantize}
\end{equation}
 Since the
quantity  $\int_0^L
\sqrt{\frac{g(x, \lambda)}{D(x)}}dx $ is a monotone
function of $\lambda$ whenever $g(x, \lambda)$ is positive it follows
that there is an increasing sequence of (approximate) eigenvalues
asymptotic to $\alpha_i^-$ in each interval $(\alpha_{i-1},\alpha_i)$.   

These heuristics motivate the following version of the
classical Sturm theorem.

\begin{thm}  \label{thm:Sturm}
Consider the rational Sturm Liouville pencil subject to Neumann boundary conditions
\[
-(D(x) p_{x})_x = \lambda W_0(x) p - \sum_{i=1}^N \frac{W_i(x)}{\lambda - \alpha_i}
- V(x) p = g(x, \lambda) p, \qquad
p_x(a) = 0 = p_x(b)
\]
with $D(x),V(x),W_i(x)$ continuous on $[a,b]$ and $D(x),W_i(x) > 0$ on
$[a,b]$. The quantities $\{\alpha_i \}_{i=1}^N$ will be
referred to as the singular values. For ease of notation we will let $\alpha_0=-\infty$ and
$\alpha_{N+1}=+\infty$, and will let $I_j,~~~j \in 0\ldots N$ denote the interval $I_j =
(\alpha_{j},\alpha_{j+1}).$ From the previous result the eigenvalues
are real and semi-simple. Additionally the following oscillation results hold.
\begin{itemize}
\item {\bf Sturm Theorem}  For each integer $k\geq 0$ there exists one
  eigenvalue $\lambda_k^{j}$ in each interval $I_j= (\alpha_{j},\alpha_{j+1})~~~j\in 0 \ldots N$
  such that the corresponding eigenfunction  $p_k^{j}(x) $ has exactly $k$ zeroes
  in the open interval $(a,b)$. 
\item By definition we have that $\lambda_k^{j} >
  \lambda_k^{j'}$ if $j>j'$. Further we have that $\lambda_k^{j}  >
  \lambda_{k'}^{j}$ if $k>k'$. 
\item Each $\alpha_i$ except $\alpha_0$ is an accumulation point of
  eigenvalues, so $\alpha_i, \, i \in 1 \ldots N$ are in the
  essential spectrum.  
\end{itemize}

\end{thm}
In other words we have a Sturm oscillation type theorem that holds in
every interval $(\alpha_i,\alpha_{i+1})$: the eigenvalues in each such
interval are increasing and  indexed by the number of roots of the
corresponding eigenfunction in $(a,b)$.

\begin{remarks}
We first note that the
fact that the pencil has finite accumulation points of eigenvalues
means that the original operator is neither compact nor does it have compact
resolvent. This illustrates that the problem in which some
species do not diffuse is a singular perturbation of the problem in which
those same species have a small diffusion coefficient, since in that case it
is straightforward to see that the operator has compact resolvent. 

We also note that the more general operator 
\[
-\frac{d}{dx} \left( D(x) u_x\right)  = \lambda u - \sum_{i=1}^N\frac{W_i(x)}{\lambda - \alpha_i(x)}
\]has only real eigenvalues by the argument given previously but has essential spectrum for $\lambda
\in \cup_{i=1}^N \ran(\alpha_i(x))$. This type of operator has been
analyzed by Adamjan, Langer and Langer
\cite{Adamjan.Langer.Langer.2001}  for $N=1$, including construction
of the spectral measure corresponding to the essential spectrum. 
 \end{remarks}
 
\begin{proof}
The proof hinges on the fact that the (eigenvalue dependent) potential
$g(x, \lambda)$ has the property that 
\begin{itemize}
\item $g(x,\lambda) \rightarrow -\infty$ as $\lambda \rightarrow
  \alpha_i^+$
\item $g(x,\lambda) \rightarrow +\infty$ as $ \lambda \rightarrow
  \alpha_{i+1}^-.$
\item $\frac{\partial g}{\partial \lambda} > 0,$
\end{itemize}
which are essentially the conditions needed to apply the classical
Sturm oscillation theorem for a one dimensional second order
operator. For the convenience of the reader a proof is given in Appendix 
\ref{app:sturm}. 
\end{proof}

\begin{example} \label{sec:Example3p9}
To illustrate this theorem we consider the rational Sturm-Liouville
pencil
\[
-u_{xx} + \sin(x) u = \lambda u - \frac{.2 + \cos^2(x)}{\lambda-2} u,
\qquad u(0)=0=u(\pi)
\] 
which is equivalent to the system 
\begin{align}
&-u_{xx} + \sin(x) u + v = \lambda u; \qquad u(0)=0=u(\pi) \nonumber \\
& \alpha u + (.2 + \cos^2(x)) u = \lambda v  
\label{eqn:SturmLiouvilleSystem}
\end{align}

\begin{figure}[htbp]
\begin{center}
\includegraphics[width=\textwidth]{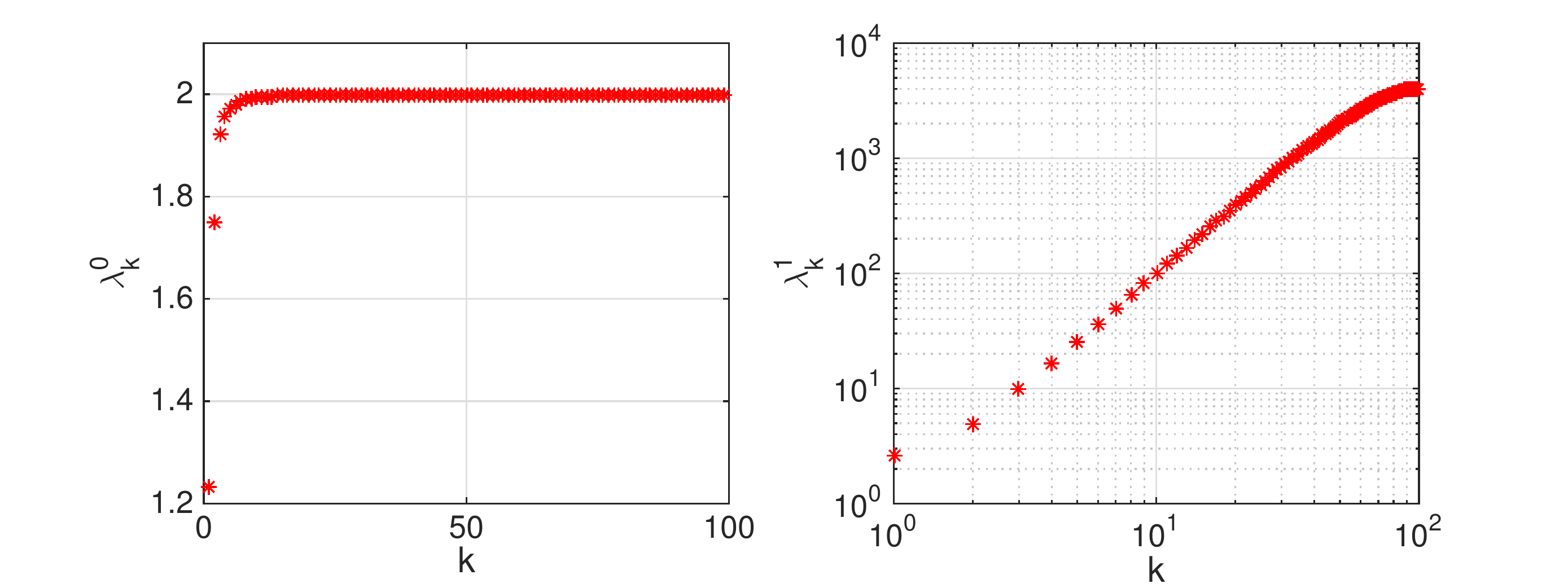}
\caption{Eigenvalues for Eqn. \eqref{eqn:SturmLiouvilleSystem}}
\label{fig:Ex3p9_evalues}
\end{center}
\end{figure}
There is one singular value, $\alpha_1=2$, and two open intervals on
the real line for which ${\bf H}(\lambda)$ is defined,
$(-\infty,2)$ and $(2,+\infty)$. We solved the system
(\ref{eqn:SturmLiouvilleSystem}) numerically as described in Appendix \ref{app:num}.

 Using the discretization described with $n_x=100$ (see Appendix \ref{app:num} for details), we will
 have 99 eigenvalues in the interval $I_0 = (-\infty,2)$, denoted
 $\lambda^0_k$, and 99 eigenvalues in the interval $I_1 = (2,\infty)$,
 denoted $\lambda^1_k$. The first set ($\lambda^0_k$) will have an
 accumulation point at $\alpha_1=2$, the second set ($\lambda^1_k$) at
 $\alpha_2= \infty$; this is shown in
 Fig. \ref{fig:Ex3p9_evalues}. The corresponding eigenfunctions we
 will index in the same way: i.e. as $u^j_k$ for the $k$th eigenvalue
 in interval $I_j$. We compare these eigenvalues with the WKB
 prediction $\frac{1}{\pi}\int_0^\pi \sqrt{\lambda - \sin(x) -\frac{.2
     + \cos^2(x)}{\lambda-2}} dx = k$, which we also evaluated
   numerically. 

The numerical values of
eigenvalues for $k\in 0 \ldots 4$ are:
\begin{center}
\begin{tabular}{||c||c|c|c|c|c||}
\hline\hline
$\lambda$ Interval & $k=0$ & $k=1$ & $k=2$ & $k=3$ & $k=4$ \\\hline
$(-\infty,2)$ (Numerical) & 1.22 & 1.75 & 1.92 & 1.95 & 1.97 \\\hline
$(-\infty,2)$ (WKB) & 1.029 & 1.773 & 1.916 & 1.956 & 1.973 \\\hline
$(2,\infty)$  (Numerical) & 2.59 & 4.88 & 9.65 & 16.53 & 25.41\\\hline
 $(2,\infty)$  (WKB) & 2.68 & 4.88 & 9.73 & 16.69 & 25.67\\\hline
\end{tabular}
\end{center}

The corresponding eigenfunctions,
$u^0_k$ and $u^1_k$, are shown in Fig.  \ref{fig:efunctions}. The left column shows the primary ($u^0_k$) and auxiliary eigenfunctions ($v^0_k$) for the interval $I_0$, while the right column shows $u^1_k$ and $v^1_k$. Note that $u^j_k$ has exactly $k$ zeros, as expected, and that the nodal sets of $u^j_k$ and $v^j_k$ coincide for each $j,k$. The functions $u^0_k$ and $v^0_k$ are of opposite sign, since $\lambda^0_k < 2$, while $u^1_k$ and $v^1_k$ are same-signed.
\begin{figure}[htbp]
\begin{center}
\includegraphics[width=0.5\textwidth]{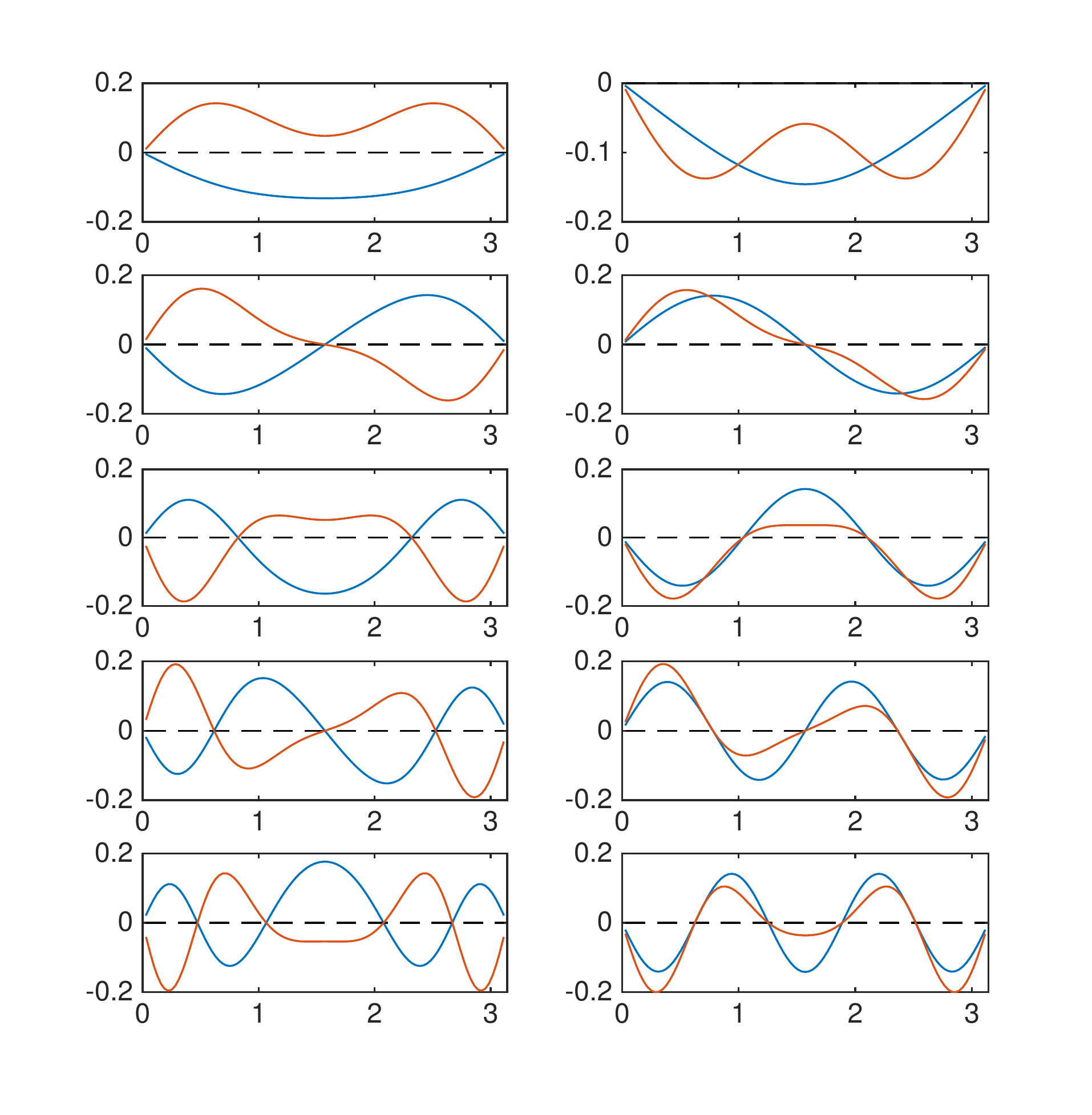}
\caption{Left column: Eigenfunctions corresponding to $\lambda \in (-\infty, 2)$. Both primary $u$ (blue) and auxiliary $v$ functions (red) are shown. From top to bottom: $u^0_k$  and $v^0_k$, for $k=0,1,2,3$ and $4$. Right column: same, but for $\lambda \in (2,\infty)$; i.e.  $u^1_k$  and $v^1_k$ }
\label{fig:efunctions}
\end{center}
\end{figure}

%
%

\end{example}

We next use
Eqn. \eqref{eqn:WKB}, \eqref{eqn:Quantize} to make the following prediction for the asymptotics of the
eigenvalues accumulating near $\lambda=2$:
\[
\lambda_k^0 \approx 2- \frac{\left(\frac{1}{\pi}\int_0^\pi \sqrt{.2 +
      \cos^2(x)}\right)^2}{k^2}+ O(k^{-3}) \approx 2-\frac{.649}{k^2},
\]
while the other branch of eigenvalues has the asymptotics
\[
\lambda_k^1 \approx k^2 + O(1)
\]
The plots in Figure \ref{fig:evalues_WKB} depicts a graph of $(\lambda_k^1-2)k^2$
vs $k$ for the eigenvalues in the region  $(-\infty,2)$ together with a plot
of  $\lambda^2_k k^{-2}$ vs. $k$ for the eigenvalues in the region
$(2, \infty)$. These are expected to asymptote to $.649$ and $1$
respectively. This is repeated for three different values of $\Delta
x$, the discretization size, in order to assess the role of
discretization error. In each case we see that the eigenvalue curve
hugs the asymptote for a while before falling away due to
discretization error, and that as the size of the discretization is
decreased the curve hugs the presumed asymptotic value for longer, as
one would expect. 

\begin{figure}[htbp]
\begin{center}
\includegraphics[width=\textwidth]{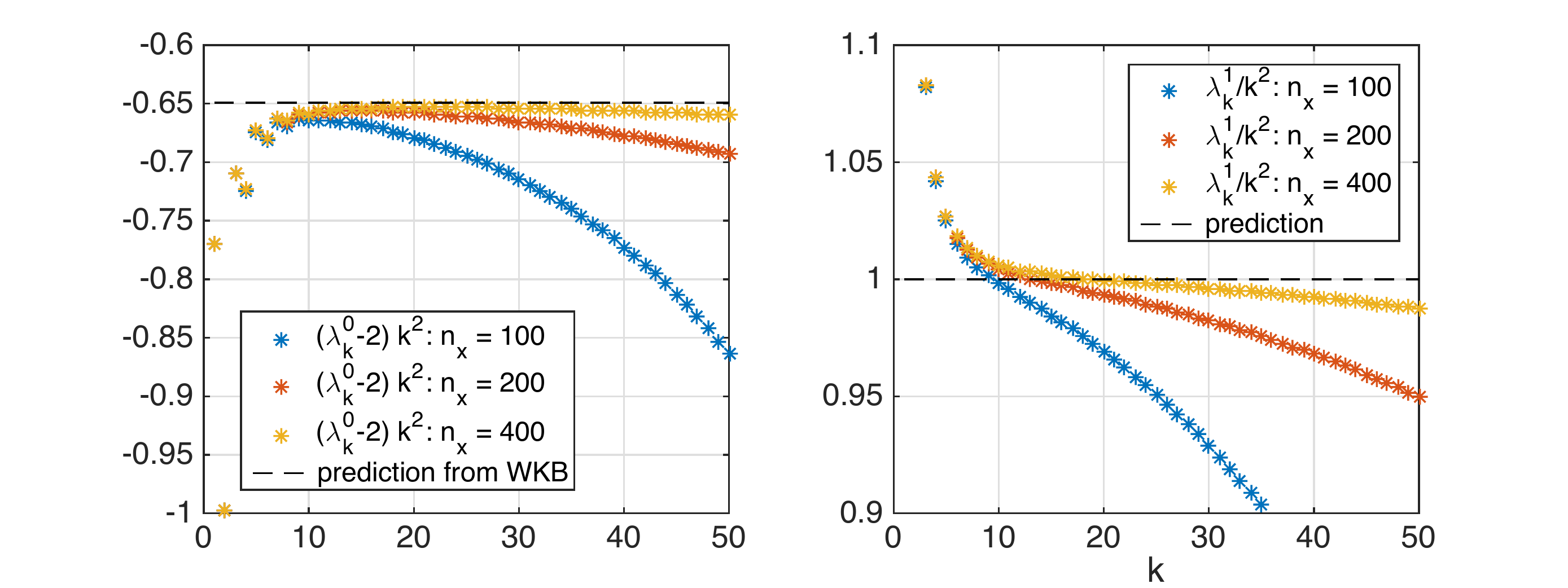}
\caption{Asymptotics for the eigenvalues of Eqn. \eqref{eqn:SturmLiouvilleSystem}. Left: $(\lambda^0_k - 2) k^2$, confirming a WKB calculation for eigenvalues that accumulate at $\alpha_1 = 2$. Right: $\lambda^2_k/k^2$ for eigenvalues that accumulate at $\alpha_2 = \infty$, showing $k^2$ scaling.}
\label{fig:evalues_WKB}
\end{center}
\end{figure}

\section{Application: Spatial models in epidemiology}
\label{sec:murray_model}
We now consider an example with detailed numerics. 

\vspace{0.4cm}

\begin{tabular}{|l|c|}
	\hline
\textbf{Variable} & \textbf{Unit}  \\
	\hline \hline
$S$ susceptible host density & foxes/km$^2$ \\
$E$ exposed host density & foxes/km$^2$ \\
$I$ infected hosts density & foxes/km$^2$ \\
$t$ time & year \\
\hline
\textbf{Parameter} & \textbf{Value}  \\
	\hline \hline
$a$ : average birth rate           & 1 per year \\ 
$b$ : average intrinsic death rate & 0.5 per year \\
$1/\alpha(x)$ : average duration of clinical disease & $\approx 5$ days \\
$1/\sigma$ : average incubation period & 28 days \\
$\beta(x)$ : disease transmission coefficient & $\approx 80$ km$^2$ per year \\ 
$K$ : carrying capacity                  & 0.25 to 4 foxes per km$^2$ \\
$D$ : diffusion coefficient & 50  to 330 km$^2$ per year\\
	\hline
\end{tabular}

\vspace{0.4cm}

In \cite{murray1986} the following one-dimensional model for the spread of rabies among foxes was considered.
\begin{align*}
& \frac{\partial E}{\partial t} = \beta(x) I S - \sigma E - \left( b + (a-b)\frac{N}{K}\right) E& \\
&\frac{\partial I}{\partial t}=\frac{\partial}{\partial x} \left( D(x) \frac{\partial I}{\partial x} \right) + \sigma E - \alpha(x) I -\left(b + (a-b)\frac{N}{K}\right)I& \\
&\frac{\partial S}{\partial t} = (a-b) \left(1-\frac{N}{K}\right) S-\beta(x) I S& \\
&\frac{\partial I}{\partial \eta}|_{\partial \Omega}=0,&
\end{align*}
where $S, E$ and $I$ are the population densities of the susceptible, exposed and infectious foxes,
respectively and $N=S+E+I$ is the total fox population. 
When susceptible hosts contact infectious ones, they become exposed at rate $\beta(x)$.
Exposed hosts remain in that class for an average period of $\sigma^{-1}$ before they
transition into the infectious class. Hosts remain infectious for an average period 
of $\alpha(x)^{-1}$ before they die. 
While infectious, hosts experience random wandering; this movement is modeled as diffusion with coefficient $D(x)$.  
All host classes experience average birth rate $a$
and death rate $b$. They also compete for resources and so it is assumed that the
environmental carrying capacity is $K$.       
Typical parameter values are given in the table above. 

The authors \cite{murray1986} assumed that $\alpha,\beta, D$ are positive continuous functions 
depending on the spatial variable $x$.  In this case there is the disease-free
steady state $(E=0, I=0, S=K)$. Linearizing about the steady state gives the 
eigenvalue problem 
\begin{align}
&-(\sigma+a) E + \beta(x) K I = \lambda E& \nonumber\\
& \frac{\partial}{\partial x} \left( D(x) \frac{\partial I}{\partial x} \right) + \sigma E - (\alpha(x)+a)I = \lambda I &  \label{eqn:Murray} \\
& \frac{\partial I}{\partial x}  |_{\partial \Omega} = 0. \nonumber& 
\end{align}
Algebraically eliminating $E$ gives the eigenvalue pencil 
\[
 \frac{\partial}{\partial x} \left(D(x) \frac{\partial I}{\partial x} \right) -(\alpha(x)+a) I = \left( \lambda - \frac{\sigma \beta(x) K }{\lambda + \sigma + a } \right) I = g(\lambda,x) I \qquad \frac{\partial I}{\partial x} |_{\partial \Omega} = 0.
\]
Our first observation is that the function $g(\lambda, x) = \lambda - \frac{\sigma \beta(x) K}{\lambda + \sigma + a}$ is a Herglotz function for all $x$ as 
long as $\sigma,\beta(x), K, a>0$, and thus the eigenvalues of this problem are all real.

This model was later considered by Wang and Zhao\cite{wang2012}, who used the 
next-generation operator approach \cite{vanden2002} to give a recipe for 
computing the reproduction number, $R_0$, and understanding its
dependence on parameters in the spatially heterogeneous problem. 
The basic reproductive number $R_0$ is an important epidemiological threshold,
since it determines whether the disease can invade the population
(with $R_0>1$ indicating that the disease-free steady state is unstable)
or not (with $R_0<1$ indicating that the disease-free steady state is locally asymptotically stable).  

Defining $\tilde{\lambda}=-\lambda$, we see that the eigenvalue problem directly above can be written as:
\begin{eqnarray}
 -\frac{d}{dx} \left( D(x) \frac{d I}{dx} \right)  +(\alpha(x)+a) I + \frac{\sigma \beta(x) K }{\tilde{\lambda} - (\sigma + a) }& =&  \tilde{\lambda}I ;  \;  \frac{d I}{dx} \vert_{\partial \Omega} = 0 \label{Wang_Eprb_orig}
  \end{eqnarray}
 to which we can immediately apply Theorem \ref{thm:Sturm} to see that this has a smallest real eigenvalue which is simple and which will govern stability of Eqn. \eqref{eqn:Murray} (this is Lemma 4.1 from \cite{wang2012}). 
 
Wang and Zhao analyze $R_0$ through a related eigenvalue problem,
\begin{equation} -\frac{d}{dx}\left( D(x) \frac{d\phi}{dx}\right) + (\alpha+a)\phi = \mu \frac{\sigma K \beta(x)}{\sigma+a}\phi;
 \label{Wang_Eprb_assoc}
\end{equation}
they show that the reproduction number (for Eq. \eqref{Wang_Eprb_orig})
is the inverse of the principal eigenvalue of
Eqn. \eqref{Wang_Eprb_assoc}: i.e. $R_0 = \frac{1}{\mu_1}$.
However, this requires us to solve a \textit{second} eigenvalue problem, which is not obviously identical to the first.  

We next show that there is a relationship between $\lambda$ and
$R_0$, which eliminates our need to solve the second eigenvalue
problem. The basic idea is to note that Eq \eqref{Wang_Eprb_orig} with
$\tilde \lambda=0$ is the same as Eq \eqref{Wang_Eprb_assoc} with
$\mu=1$. It follows from the Sturm oscillation theorem that if $\mu=1$
is the $k^{th}$ eigenvalue of \eqref{Wang_Eprb_assoc}  then
$\lambda=0$ is the  $k^{th}$ eigenvalue of \eqref{Wang_Eprb_orig}, because the (identical) eigenfunctions must have the same number of zeros. Moreover, $\mu=1$
lies between the $k^{th}$ and $(k+1)^{st}$ eigenvalues of
\eqref{Wang_Eprb_assoc} if and only if $\tilde \lambda=0$ lies between
the $k^{th}$ and $(k+1)^{st}$ eigenvalues of
\eqref{Wang_Eprb_orig}. In conclusion, solutions to \eqref{Wang_Eprb_orig} are stable if the
lowest eigenvalue of \eqref{Wang_Eprb_assoc}  satisfies $\mu_1>1$. 

\begin{figure}[htbp]
\begin{center}
\includegraphics[width=0.32\textwidth]{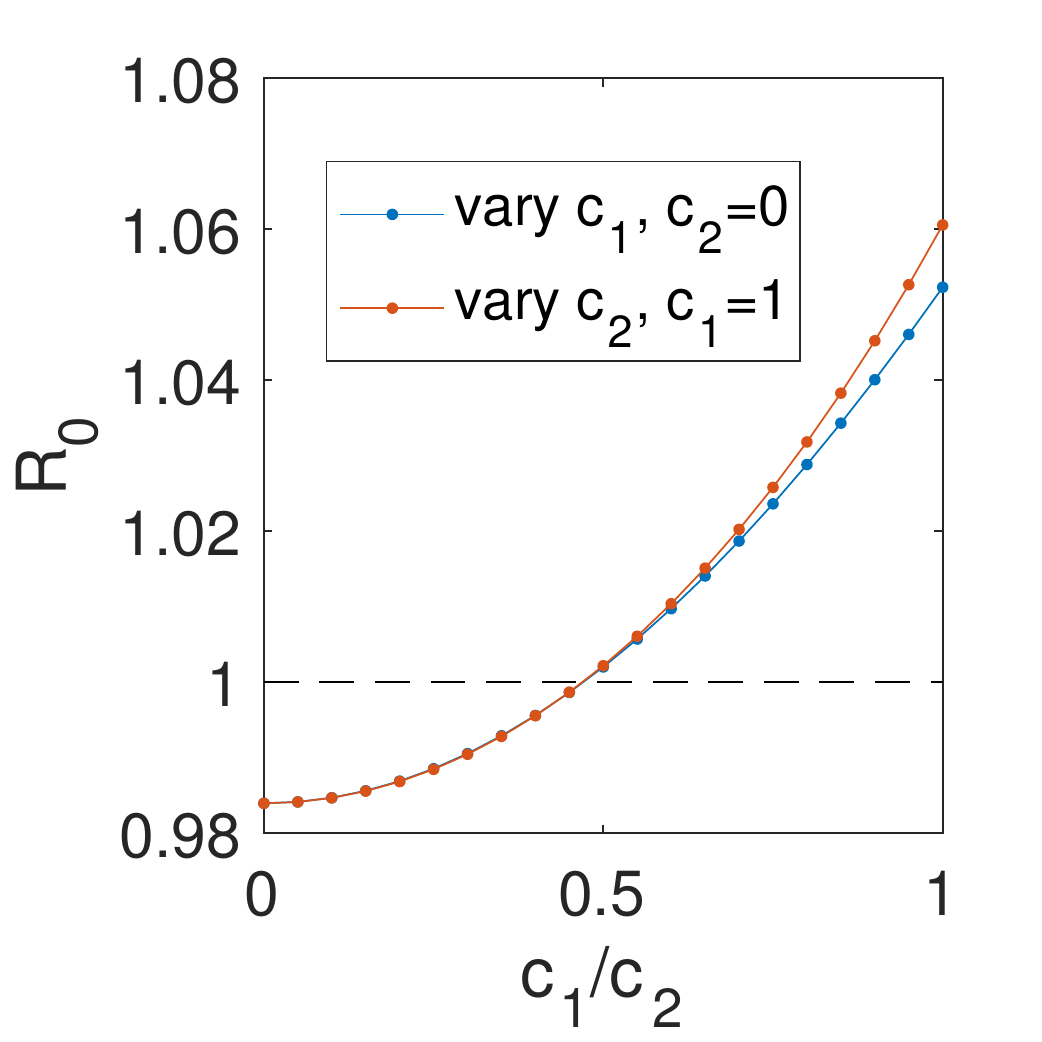}
\includegraphics[width=0.32\textwidth]{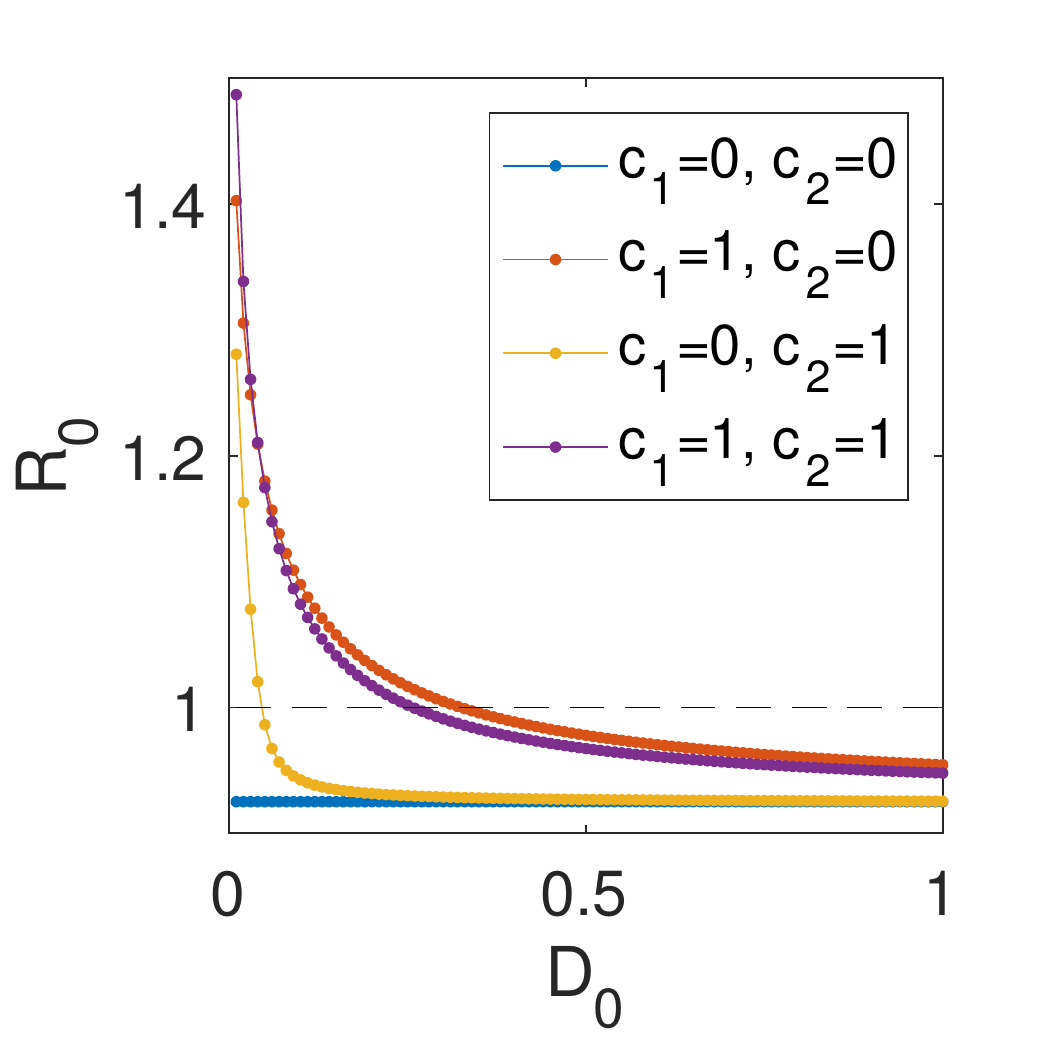}
\includegraphics[width=0.32\textwidth]{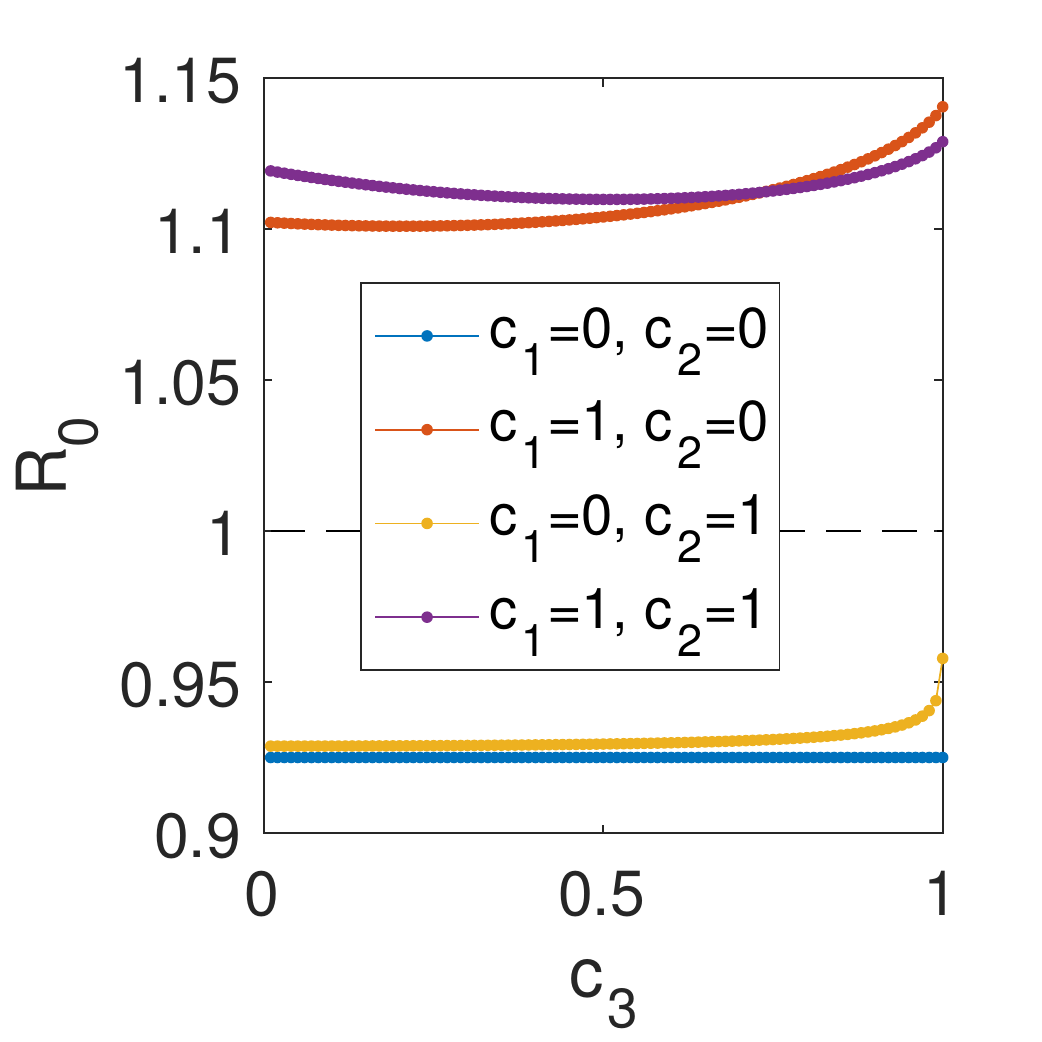}
\caption{The impact of spatial heterogeneity on reproduction number. (A) For fixed $[\beta]$, $R_0$ is minimized for the homogeneous case; similarly for fixed $[\alpha]$. (B) When $\beta$ and $\alpha$ are spatially homogeneous (i.e. $\beta = [\beta]$ and $\alpha = [\alpha]$), the baseline level of diffusion, $D_0$, has no impact on $R_0$. $D(x) = D_0(1 + \cos(\pi x))$; see text for other parameters. (C) When $\beta$ and $\alpha$ are spatially homogeneous, the level of spatial heterogeneity of diffusion has no impact on $R_0$. $D(x) = 0.0685(1 + c_3 \cos(\pi x))$; see text for other parameters.}
\label{fig:ReproWangFig1}
\end{center}
\end{figure}

We next reproduce some of Wang and Zhao's numerical results. One such finding concerns the impact of spatial heterogeneity; 
in particular that $\mu_1 \le \bar{\mu}_1$, where ${\mu}_1$ is the principal eigenvalue of their associated eigenvalue problem, and $\bar{\mu}_1$ is the principal eigenvalue where $\beta$ and $\alpha$ are replaced by their spatial averages (this is Lemma 4.4 in \cite{wang2012}). Since $R_0 = 1/\mu_1$, this means that $R_0$ is minimized for the spatially homogeneous problem. 
We used the constant parameters: $a = 0.0027$, $\alpha=0.2$, $\sigma=0.0357$, $K=0.98$. Then $\beta(x) = 0.2192(1+c_1 \cos(\pi x))$.
Fig. \ref{fig:ReproWangFig1}A shows that $R_0$ is minimized when $c_1=0$. Similarly, restoring $\beta = 0.2192$ and varying $\alpha(x) = 0.2(1+c_2\cos(\pi x))$, $R_0$ is minimized for the spatially homogenous case.

We next look at the effect of varying $D_0$, the baseline diffusion coefficient, when spatial heterogeneity in $\beta$ and/or $\alpha$ are present. 
To get a fair comparison for the role of heterogeneity, we used $\beta(x) = 0.2192 (1+ c_1 \cos(\pi x))$ and $\alpha(x) = \frac{2}{\pi}(1+ c_2 \left(\sin(\pi x)-\frac{2}{\pi})\right)$, so that $[\alpha]$ is preserved as $c_2$ changes; the variable diffusion is: $D(x) = D_0(1 + \cos(\pi x))$.  Remarkably, $D_0$ has no effect on $R_0$ when both $\beta$ and $\alpha$ are constant.

We can also vary the modulation magnitude of the diffusion, i.e. allow $D(x) = 0.0685(1 + c_3 \cos(\pi x))$, where $c_3$ is varied. Wang found that again $R_0$ is constant when $\beta$ and $\alpha$ are spatially homogeneous. We confirm this in Fig. \ref{fig:ReproWangFig1}C, again including heterogeneity in either $\beta$, or $\alpha$, both, or neither.  

Wang and Zhao next considered the problem of designing a vaccine strategy, based on the hypothesis that the impact of the vaccine could be modeled by changes to disease transmission $\beta(x)$.  They suppose that
\[ \beta(x) = 6x(1-x) \]
and consider how $R_0$ would change if we modified 
\[ \beta \rightarrow \beta \frac{1}{1+v_0(x)} \]
where
\[ v_0 (x)  = \left\{ \begin{matrix} \frac{c_0}{L} & 0 \le a_0 \le x < a_0+L \le 1\\
						0 & {\rm otherwise}
		\end{matrix} \right.
		\]
The authors interpret $\frac{v_0}{1+v_0} = 1-\frac{1}{1+v_0}$ as the vaccine efficacy.  The spatial function can be motivated by the idea of distributing a total quantity of $c_0$ over an area of size $L$, beginning at $a_0$. 

\begin{figure}[htbp]
\begin{center}
\includegraphics[width=\textwidth]{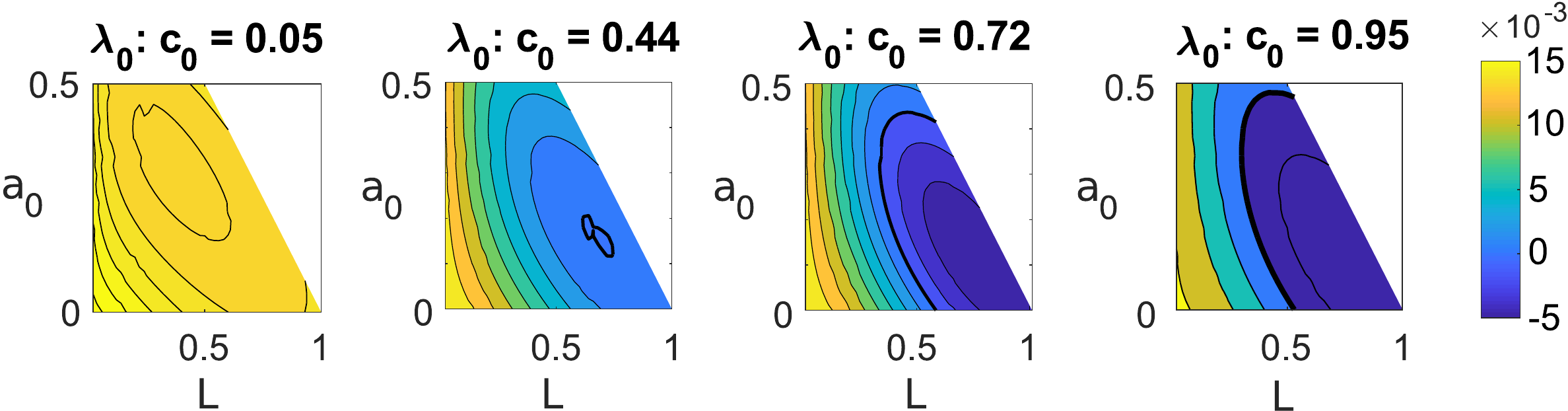}
\caption{Optimizing a vaccine strategy to control rabies. Principal eigenvalue $\lambda_0$, indicating whether an initial infection will spread ($\lambda_0 > 0$) or dissipate ($\lambda_0 < 0$). In each panel, $c_0$ is fixed while $a_0$ and $L$ are varied. From left to right: $c_0 = 0.05, 0.44, 0.72, 0.95$. }
\label{fig:ReproWangFig2}
\end{center}
\end{figure}

To understand the impact of a vaccine strategy, we solved Eqn. \eqref{Wang_Eprb_orig} by sweeping over the parameters $c_0$, $a_0$ and $L$. We surveyed $0 \le c_0, L \le 1$; by symmetry of $\beta(x)$ it suffices to consider $0 \le a_0 \le 0.5$, and we limited our consideration to $a_0 + L \le 1$.  
Other parameters were as used previously: $a = 0.0027$, $b=a/2$, $\sigma=0.0357$, $K=1.5$, $\alpha(x) = \alpha_0=0.2$, $D(x) = D_0 = 0.1371$. 

First, we can ask how much vaccine is required to prevent an epidemic; i.e. how large does $c_0$ need to be for $\lambda_0 < 0$.  Not surprisingly, $\lambda_0$ decreases with $c_0$; the minimum value of $c_0$ for which stability was observed was $0.44$.  
However, stability also depends strongly on where the vaccine is used, as well as its spatial distribution: in Fig. \ref{fig:ReproWangFig2}
we show $\lambda_0$ as a function of $a_0$ and $L$, for selected values of $c_0$: for $c_0 \ge 0.44$ (three rightmost panels), the boundary of the stable region is shown as bold. Note that even for the largest vaccine quantity tested ($c_0=0.95$), a suboptimal distribution strategy can easily fail to contain the infection (i.e. $\lambda_0 > 0$).

Finally, we can examine the optimal vaccine strategy --- the configuration that minimizes $\lambda_0$, given a total quantity $c_0$. 
For the values of $c_0$ shown in Fig. \ref{fig:ReproWangFig2}, we found the $(a_0,L)$ values at which $\lambda_0$ was minimized, and show the corresponding vaccine strategy and its effect on the disease transmission rate (i.e. the modified $\beta(x)$) in Fig. \ref{fig:ReproWangOptVacStrat}. In all cases, the optimal strategy is symmetric; the vaccine should always be distributed on the area where $\beta$ is largest. 

\begin{figure}[htbp]
\begin{center}
\includegraphics[width=\textwidth]{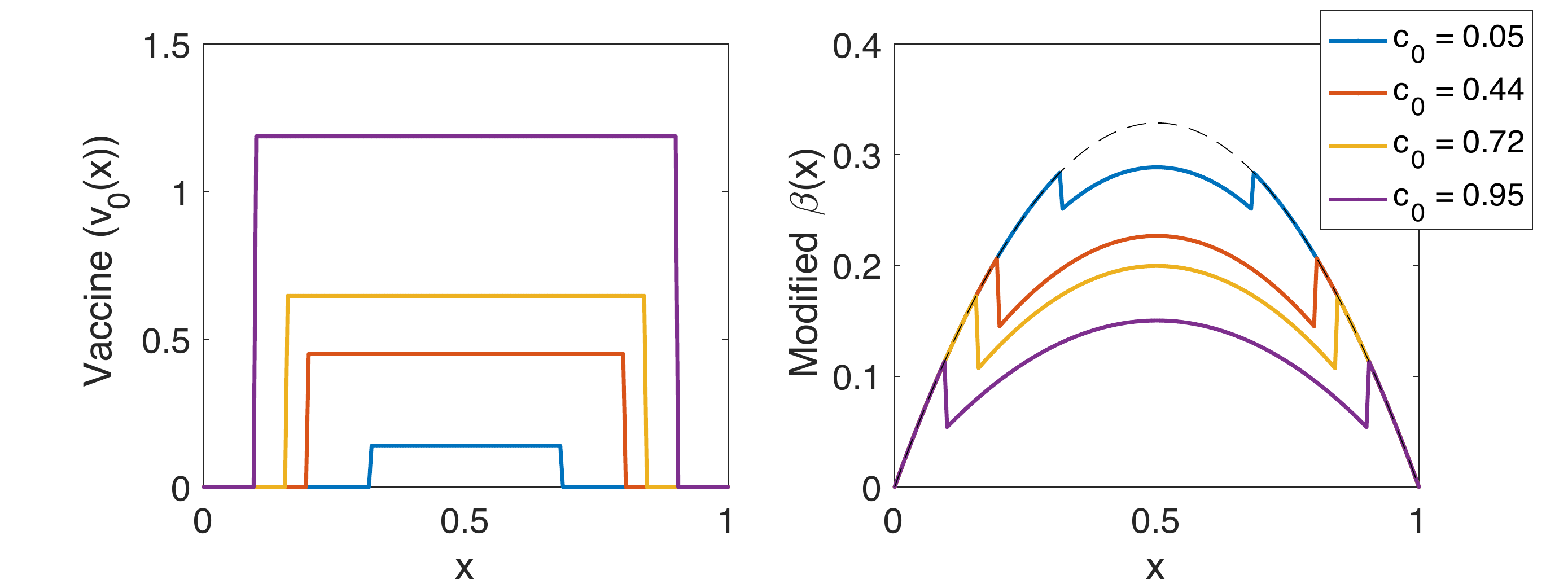}
\caption{The optimal vaccine strategy to control rabies, given a total quantity of available vaccine $c_0$. Left panel: The optimal strategy obtained by minimizing over dispersal location $a_0$ and dispersal area $L$. Right panel: the resulting disease transmission coefficient, showing a decrease in the dispersal area. For reference, $\beta(x)$ with no vaccine is shown (black dashed line).}
\label{fig:ReproWangOptVacStrat}
\end{center}
\end{figure}

\section*{Acknowledgments}
The authors are grateful for support from the University of  Illinois Research 
Board through grant RB 17060 to Z.R., from the National Science
Foundation through DMS-1615418 to J.C.B.,  

\begin{appendices}

\section{Proof of the Sturm oscillation theorem for rational Herglotz
  pencils}
\label{app:sturm}

We give a short proof of the Sturm property for the rational Herglotz
pencil problem 
\[
-\partial_x(D(x)p_{x}) = \left( -V(x) + \sum_{i=1}^N \frac{\beta_i(x)}{\lambda - \alpha_i}
+ \lambda \right) p = g(x,\lambda) p \qquad p_x(a) = 0 = p_x(b) 
\] 
We take the convention that $\alpha_0=-\infty$ and
$\alpha_{N+1}=+\infty$. Then in each interval
$(\alpha_{j-1},\alpha_j),~~~j\in\{1\ldots N\}$ there exists a strictly increasing sequence of (real)
eigenvalues $\{\lambda^{j}_k\}_{k=0}^\infty \in
(\alpha_{j-1},\alpha_j)$ whose only accumulation point is $\alpha_j$
such that the eigenfunction $\phi_k^{j}(x)$ has exactly $k$ (simple)
roots in the interval $(a,b)$. The exact boundary conditions used are 
not particularly important -- the same basic arguments work with
Dirichlet, Robin or (with some additional modifications) periodic
boundary conditions -- but we work with Neumann boundary conditions
since these are most appropriate for the models considered here. 

The easiest way to prove the Sturm theorem is via the Pr\"ufer angle, a change
of variable for the second order linear differential equation defined
by 
\begin{align}
& p(x,\lambda) = R(x,\lambda) \cos(\Theta(x,\lambda))
\label{eqn:prufer} \\
& p_x(x,\lambda) = R(x,\lambda) \sin(\Theta(x,\lambda))
\label{eqn:prufder}. 
\end{align}   
It is well-known (at least to those that well-know it) that if 
$p(x,\lambda)$ satisfies the second order equation 
\[
-\partial_x (D(x) p_{x}) = g(x,\lambda) p
\] then 
the
Pr\"ufer angle $\Theta(x,\lambda)$ satisfies the equation 
\begin{equation}
\Theta_x = - \sin^2 \Theta - g(x,\lambda) \cos^2 \Theta -\frac{D_x
  R}{D} \sin \Theta \cos \Theta 
\label{eqn:Angle}
\end{equation}
while the radial function $R(x,\lambda)$ satisfies the equation 
\begin{equation}
R_x= \left(1-\frac{g}{D} \right) R \sin(\Theta) \cos(\Theta)-\frac{D_x}{D} R \sin^2
\Theta. 
\label{eqn:Radius}
\end{equation}
We will mostly work with the angle $\Theta(x,\lambda)$: the radial
function $R(x,\lambda)$ is important only insofar as equation (\ref{eqn:Radius})
shows that $R(x,\lambda)$ cannot vanish unless it
is identically zero, implying that the Pr\"ufer coordinate change is
well-defined for all $x$. 

The basic idea of the proof is that the winding of the Pr\"ufer angle
is a good measure of the winding of the eigenfunctions, and that this
can be understood via equations (\ref{eqn:prufer}, \ref{eqn:prufder}). 
Since $R(x,\lambda)$ cannot vanish, the roots of $p$ correspond to places where
$\Theta(x,\lambda)=k \pi$ while critical points of $p$ correspond to places where 
$\Theta(x,\lambda)=(k+\frac12)\pi$. We will look in particular at the
solution to (\ref{eqn:Angle}) satisfying the initial condition
  $\Theta(a,\lambda)=0$ (and, for concreteness sake, $R(a,\lambda)=1$), so 
that the corresponding $p$ satisfies $p(a,\lambda)=1,
p_x(a,\lambda)=0.$
In this setting eigenvalues will correspond to values of $\lambda=\lambda^*$
such that $\Theta(b,\lambda^*)=k \pi.$ 

We note several facts. First note that, for fixed values of $x$ the
function $\Theta(x,\lambda)$ is a decreasing function of $\lambda$. 
This follows from differentiating equation \ref{eqn:Angle} with
respect to $\lambda$, using the fact that $\frac{d g}{d\lambda}>0$ and
applying the Gronwall inequality. Next we note that, for $\lambda$
sufficiently close to $\alpha_{j-1}^+,$ $\Theta(x,\lambda)$ must
satisfy $0\leq \Theta(x,\lambda) < \frac{\pi}{2}$, with
$\Theta(x,\lambda)=0$ if and only if $x=0$. To see this note that we
can choose $\lambda$ close enough to $\alpha_{j-1}$ to guarantee that
$g(x,\lambda)<0$. In this case we have that $\Theta_x|_{\Theta=0}>0$
and $\Theta_x|_{\Theta=\frac{\pi}{2}}<0$. This guarantees that the
interval $\Theta \in (0,\frac{\pi}{2})$ is invariant in positive
$x$. Finally note that, if $\Theta = -(k+\frac12)\pi$ then
$\Theta_x=-1$, so the curve $\Theta(x,\lambda)$ always crosses the
line $-(k+\frac12)\pi$ transversely and in the same direction.

The proof now follows from considering a homotopy in $\lambda$: we
begin with the curve $\Theta(x,\lambda)$ with $\lambda$ sufficiently
close to $\alpha_{j-1}$ so that $\Theta(x,\lambda)$ lies in the
interval $(0,\frac{\pi}{2})$. Thus no value of $\lambda$ in this range
can be an eigenvalue since $\Theta(b,\lambda)$ is strictly between $0$
and $\frac{\pi}{2}$, so we are below the spectrum.  

Next we consider increasing $\lambda$, which will decrease
$\Theta(b,\lambda)$. The first eigenvalue occurs when
$\Theta(b,\lambda)=0.$
Call this value $\lambda_0$. While we do not know much about the shape
of the curve $\Theta(x,\lambda_0)$ we do know that it cannot cross the
line $\Theta=-\frac{\pi}{2}$ since at any point where
$\Theta(x,\lambda)=-(k+\frac12)\pi$ we have that $\Theta_x=-1$. 
Hence as $\lambda$ is increased the first time that the curve
$\Theta(x,\lambda)$ can cross the line $\Theta=-(k+\frac12)\pi$ must be
at the endpoint $x=b$. 

This pattern continues as $\lambda$ increases:
$\Theta(b,\lambda)$
decreases and alternately crosses the lines $\Theta=-k\pi$ and
$\Theta=-(k+\frac12)\pi$: the former correspond to eigenvalues and at
the latter $p(x,\lambda)$ picks up a root. Since $R(x,\lambda)$ is
non-vanishing and $\Theta$ and $\Theta_x$ cannot vanish simultaneously
it follows that the $k^{th}$ eigenfunction has exactly $k$ roots in
$(a,b)$. It follows from standard estimate that away from poles of
$g(x,\lambda)$ we have analytic dependence of $\Theta(b,\lambda)$ on
$\lambda$ and so solutions to $\Theta(b,\lambda)=-(k+\frac12)\pi$ 
must be isolated. It is also clear that by choosing $\lambda$
sufficiently close to $\alpha_j$ we can guarantee that $\Theta_x$ is
uniformly large and negative, and thus that $\alpha_j$ is an
accumulation point of eigenvalues. 

\section{Numerical Methods}
\label{app:num}

We use standard methods to solve the eigenvalue problem:
\begin{eqnarray}
H u + q(x) u + \sum_{i=1}^{N} \frac{\beta_i (x) u}{\lambda - \alpha_i} & = & \lambda u; \; 0 \le x \le L  \label{eqn:SLP_ev_prob}
\end{eqnarray}
where $H$ is a Sturm-Liouville operator, such as the second derivative (i.e. $Hu = -u_{xx}$).
We assume that $\beta_i(x) > 0$ on the domain $[0, L]$ and that boundary conditions are of the typical kind:
\begin{eqnarray}
b^0_0 u(0) + b^1_0 \frac{\partial u}{\partial x}(0) & = & 0 \label{eqn:leftBC} \\
b^0_L u(L) + b^1_L \frac{\partial u}{\partial x}(L) & = & 0  \label{eqn:rightBC}
\end{eqnarray}
where $b^0_0 \not= 0$ and/or $b^1_0 \not= 0$, and  $b^0_L \not= 0$ and/or $b^1_L \not= 0$.
Introducing auxiliary functions $v^{(i)} = \frac{\beta_i(x) u}{\lambda-\alpha_i}$, we find 
that Eqn. \eqref{eqn:SLP_ev_prob} is equivalent to
\begin{eqnarray}
H u + q(x) u + \sum_{i=1}^{N} v^{(i)} & = & \lambda u;\\
\beta_i (x) u + \alpha_i v^{(i)} & = & \lambda v^{(i)}, \, i=1,...,N
\end{eqnarray}
We used standard second-order discretizations to evaluate $H u$ at $n_x-1$ interior points of $[0,L]$; e.g. centered differencing for $Hu = -u_{xx}$. For the spatially variable diffusion in \S \ref{sec:murray_model}, we used the stencil
\begin{eqnarray} 
(D(x) u_x)_x  & \approx & \frac{1}{(\Delta x)^2} (D_{i-1/2} u_{i-1} - (D_{i-1/2}+D_{i+1/2}) u_i + D_{i+1/2} u_{i+1}) \nonumber \\
\end{eqnarray}
where $x_i = i \Delta x$, $u_i = u(x_i)$, $D_{i-1/2} = D(x_{i-1/2})$, etc.; i.e. the diffusion function is evaluated at half-intervals.

Finally, the boundary conditions are incorporated with the second-order stencils 
\[ \frac{\partial u}{\partial x} (0) \approx \frac{-3 u_0 + 4 u_1 - u_2}{2 \Delta x}; \qquad \frac{\partial u}{\partial x} (L) \approx \frac{3 u_{n_x} - 4 u_{n_x-1} + u_{n_x-2}}{2 \Delta x}; 
\]
This results in a linear system of size $(N+1) (n_x-1)$, representing the unknown functions $u$ and $v_i$ at each of $n_x-1$ interior points; we numerically approximate its eigenvalues using standard methods (we used the \texttt{eig} function from Matlab R2017b).

Here, we list parameters used for specific cases in the paper.
\begin{itemize}
\item In Example \ref{sec:Example3p9}, we used $L=\pi$ and $n_x = 100$; i.e. $\Delta x = \pi/n_x$ and the number of interior points is $n_x - 1=99$. For Figure \ref{fig:evalues_WKB}, we repeated the computations with $n_x = 200$ and $n_x=400$.
\item For all computations presented in \S \ref{sec:murray_model}, $L=1$ and $n_x = 200$.
\end{itemize} 
 
\end{appendices}



\end{document}